\documentclass[11pt]{amsart}
\usepackage{graphics,color}
\usepackage{amssymb,amsmath, latexsym, }
\usepackage{amssymb}
\usepackage{amsmath}
\usepackage{amsthm}
\usepackage{amsfonts}
\usepackage{mathrsfs}
\usepackage{hyperref}
\usepackage{listings}
\usepackage{color}
\usepackage[left]{lineno}
\usepackage{blindtext}
\usepackage[all,cmtip]{xy}
\usepackage{comment}
\usepackage{tikz}

\makeatletter
\def\proof{\@ifnextchar[{\@oproof}{\@nproof}}
\def\@oproof[#1][#2]{\trivlist\item[\hskip\labelsep
\textit{#2 Proof of\ #1.}~]\ignorespaces}
\def\@nproof{\trivlist\item[\hskip\labelsep\textit{Proof.}~]\ignorespaces}

\makeatother

\setbox0=\hbox{$+$}
\newdimen\plusheight
\plusheight=\ht0
\def\+{\;\lower\plusheight\hbox{$+$}\;}

\setbox0=\hbox{$-$}
\newdimen\minusheight
\minusheight=\ht0
\def\-{\;\lower\minusheight\hbox{$-$}\;}

\setbox0=\hbox{$\cdots$}
\newdimen\cdotsheight
\cdotsheight=\plusheight
\def\cds{\lower\cdotsheight\hbox{$\cdots$}}

\DeclareMathOperator{\Hom}{Hom}
\DeclareMathOperator{\Proj}{Proj}
\DeclareMathOperator{\Spec}{Spec}
\DeclareMathOperator{\Ext}{Ext}

\DeclareMathOperator{\In}{in}

\DeclareMathOperator{\Mon}{Mon}

\DeclareMathOperator{\Mini}{Min}

\DeclareMathOperator{\Tr}{Tr}
\DeclareMathOperator{\GL}{GL}

\newcommand{\ZZ}{\mathbb{Z}}

\newcommand{\RR}{\mathbb{R}}
\newcommand{\AAA}{\mathbb{A}}
\newcommand{\NN}{\mathbb{N}}
\newcommand{\Pp}{\mathcal{P}}

\newcommand{\mm}{\mathfrak{m}}
\newcommand{\pp}{\mathfrak{p}}
\newcommand{\supp}{\mathrm{supp}}
\newcommand{\height}{\mathrm{ht}}

\newcommand{\init}{\mathrm{in}}

\numberwithin{equation}{section}
\newtheorem{theorem}{Theorem}[section]
\newtheorem{lemma}[theorem]{Lemma}
\newtheorem{question}[theorem]{Question}
\newtheorem{cor}[theorem]{Corollary}
\newtheorem{proposition}[theorem]{Proposition}
\newtheorem{remark}[theorem]{Remark}

\newtheorem{definition}[theorem]{Definition}

 \newtheorem{example}[theorem]{Example}

\title{Gr\"obner deformation and $F$-singularities}
\author[]{Mitra Koley}
\address{Stat-Math Unit, Indian Statistical Institute, 203 B.T. Road, Kolkata, India 700035}
\email{mitra.koley@gmail.com}
\author[]{Matteo Varbaro}
\address{Dipartimento di Matematica, Universit\'a di Genova, Italy}
\email{varbaro@dima.unige.it}
\begin{document}
\maketitle
\begin{abstract}
For polynomial ideals in positive charachteristic, defining $F$-split rings and admitting a squarefree monomial initial ideal are different notions. In this note we show that, however, there are strong interactions in both directions. Moreover we provide an overview on which $F$-singularities are Gr\"obner deforming. Also, we prove the following characteristic-free statement: if $\pp$ is a height $h$ prime ideal such that $\init(\pp^{(h)})$ contains at least one squarefree monomial, then $\init(\pp)$ is a squarefree monomial ideal.
\end{abstract}
\bigskip
\section{Introduction}

The motivation for this note has been, essentially, trying to achieve a better understanding of the following question concerning polynomial ideals $I$ of a polynomial ring $S$ over a field $K$:

\begin{question}\label{q1}
When is there a monomial order $<$ on $S$ such that $\init_<(I)$ is a squarefree?
\end{question}

One of the reasons why this is an interesting problem arises from the recent work \cite{CV} by Conca and the second author of this paper, roughly stating that $I$ and $\init_<(I)$ are much more related than usual provided the latter is a squarefree monomial ideal. There are already many known classes of ideals $I$ (and suitable monomial orders) such that $\init_<(I)$ is squarefree, such as ideals defining Algebras with Straightening Law, Cartwright-Sturmfels ideals and Knutson ideals. In Theorem \ref{t:symb} we identify a new class: If $I$ is a radical ideal, as soon as $\init_<(I^{(h)})$ contains a squarefree monomial, where $h$ is the maximum height of a minimal prime ideal of $I$, then $\init_<(I)$ is a squarefree monomial ideal. The result is proved first in positive characteristic, and then derived over fields of characteristic 0. The proof in positive characteristic relies on the ``$F$-split'' notion and a suitable version of Fedder's criterion, see Theorem \ref{t:charp}.

\bigskip

If $K$ has positive characteristic, in which case we can speak of $F$-singularities (where $F$ stands for the Frobenius endomorphism), we investigate on the following:
\begin{question}\label{q2}
For which kind of $F$-singularities do we have that $S/I$ has those $F$-singularities provided that, for some weight vector $w\in\NN^n$, $S/\init_w(I)$ has those $F$-singularities? 
\end{question}
The two questions above are actually related: if $\init_<(I)$ is squarefree, then $S/\init_<(I)$ is $F$-split. Although there are examples of ideals $I$ such that $\init_<(I)$ is squarefree but $S/I$ is not $F$-split (e.g. see Example \ref{e2:anurag}), it turns out that $S/I$ is always $F$-injective, and even strongly $F$-injective, provided $\init_<(I)$ is squarefree (see Corollary \ref{imp}).

On the other hand, it is very easy to find examples of ideals $I$ such that $S/I$ is $F$-split but $\init_<(I)$ is not squarefree for any monomial order. However, Theorem \ref{t:charp} states that many ideals defining $F$-split rings admit a squarefree initial ideal; hence, at some extent, ``being $F$-split'' and ``admitting a squarefree initial ideal'' are connected properties.

\bigskip

Question \ref{q2} is also related to the so-called deformation problem: if $(R,\mm)$ is a Noetherian local ring and $x\in \mm$ is a nonzero divisor on $R$ such that $R/xR$ has some property $\Pp$, is it true that $R$ has property $\Pp$ as well? Of course the answer depends on the property $\Pp$, for example it is positive if $\Pp$ is ``being a domain'' and negative if $\Pp$ is ``being irreducible''. There is a fervent research on the deformation problem when $\Pp$ is an $F$-singularity, especially if ``$\Pp =$ $F$-injective'', in which case the problem is still open. Regarding Question \ref{q2}, the answers we get agree with the answers to the deformation problem; this is expected, though it needs some explanations.

\section{Gr\"obner deformations}
Throughout this note, by a ring we mean a Noetherian commutative ring with unity. A $\NN$-graded ring is a ring $R=\bigoplus_{i\in\NN}R_i$ (usually $R_0$ will be a field). A $\NN$-graded ring $R=\bigoplus_{i\in\NN}R_i$ is standard graded if $R=R_0[R_1]$.

Let $S=K[X_1,\ldots ,X_n]$ be a polynomial ring over a field $K$ and $I\subset S$ be an ideal. If $<$ is a monomial order on $S$ we can consider the initial ideal $\In_<(I)\subset S$ generated by all the monomials of the form $\In_<(f)$ with $f\in I$. It turns out that it is possible to choose a suitable weight vector $w\in(\NN_{>0})^n$ (depending on $<$ and $I$) such that $\In_<(I)=\In_w(I)$. Here $\In_w(I)$ is the ideal of $S$ generated by $\In_w(f)$ with $f\in I$, where $\In_w(f)$ stands for the sum of the terms of $f$ with maximal $w$-degree. The latter point of view is more convenient concerning some aspects. For example, besides Gr\"obner bases it also includes Sagbi bases. In fact, if $A\subset S$ is a $K$-subalgebra of $S$, consider the $K$-subalgebra $\In_<(A)\subset S$ generated by all the monomials of the form $\In_<(f)$ with $f\in A$. If $f_1,\ldots ,f_m\in A$ are a Sagbi basis of $A$, that is $\In_<(A)=K[\In_<(f_1),\ldots ,\In_<(f_m)]$, it is easy to see that $A=K[f_1,\ldots ,f_m]$. It turns out that, if $J\subset P=K[Y_1,\ldots ,Y_m]$ is the kernel of the $K$-algebra homomorphism sending $Y_i$ to $f_i$ (so that $P/J\cong A$), there exists $u\in(\NN_{>0})^m$ such that $\In_u(J)$ is the kernel of the $K$-algebra homomorphism sending $Y_i$ to $\In_<(f_i)$, hence $\In_u(J)$ is a binomial ideal and $P/\In_u(J)\cong \In_<(A)$, (cf. \cite[Corollary 2.1]{CHV}).

The formation of $\In_w(I)$ can also be seen as a deformation: Let $t$ be an extra homogenizing variable, and $\hom_w(I)\subset S[t]$ the $w$-homogenization of $I$. Then we say that $R=S[t]/\hom_w(I)$ is a {\it Gr\"obner deformation}, and we have that:
\begin{itemize}
\item $R$ is a $\NN$-graded ring such that $R_0=K$ and $t\cdot 1\in R$ has degree 1 (the grading is given by $\deg(X_i\cdot 1)=w_i$ and $\deg(t\cdot 1)=1$).
\item $t$ is a nonzero-divisor on $R$.
\item $R/tR\cong S/\In_w(I)$.
\item $R/(t-1)R\cong S/I$.
\end{itemize}
A {\it 1-parameter affine deformation over $K$} is a flat morphism $X\rightarrow \AAA^1$ where $\AAA^1$ is the affine line over $K$ and $X$ is an affine variety over $K$. In other words, a 1-parameter affine deformation over $K$ is a $K$-algebra $R$ which is a flat $K[t]$-module (equivalently a $K[t]$-module without nontrivial torsion). In the following we will write $t$ for $t\cdot 1\in R$.

\begin{lemma}
Let $R$ be a 1-parameter affine deformation over $K$. Then the following are equivalent:
\begin{enumerate}
\item $R$ is $\NN$-graded, $R_0=K$ and $t\in R$ is homogeneous of degree 1.
\item $R$ is a Gr\"obner deformation.
\end{enumerate}
\end{lemma}
\begin{proof}
We already noticed (2) $\implies$ (1). For the converse, let $V\subset R\setminus K$ be a finite dimensional graded vector space containing $t$ such that $R=K[V]$. Set $n+1=\dim_KV$. Let $v_1,\ldots ,v_n\in V$ be homogeneous elements such that $t,v_1,\ldots ,v_n$ is a $K$-basis of $V$, $S=K[X_1,\ldots ,X_n]$ and $S[z]\rightarrow R$ the $K$-algebra homomorphism sending $z$ to $t$ and $X_i$ to $v_i$. Call $J\subset S[z]$ the kernel, $I=(J+(z-1))/(z-1)\subseteq S$, $w_i=\deg(v_i)=\deg(X_i)$ and put $\deg(z)=\deg(t)=1$, so that the above map is graded. We claim that $J=\hom_w(I)$ (so that $R/tR\cong S/\In_w(I)$). So we would conclude because $R\cong S[z]/\hom_w(I)$.

\vskip 1mm

To prove the claim, it is useful to introduce the dehomogeneization homomorphism $\pi:S[z]\rightarrow S$ sending $X_i$ to itself and $z$ to 1. With this notation $I=\pi(J)$.

\vskip 1mm

Let us first see that $J\subset \hom_w(I)$. Let $F$ be a homogeneous element of $J$. We can write $F=z^rG$ where $G$ is a homogeneous polynomial of $S[z]$ not divided by $z$. Of course $\pi(F)=\pi(G)$ belongs to $I$, so $\hom_w(\pi(G))=G\in \hom_w(I)$. Since $F$ is a multiple of $G$, it belongs to $\hom_w(I)$ as well. Since $J$ is a homogeneous ideal we conclude that $J\subset \hom_w(I)$.

For the inclusion $\hom_w(I)\subset J$, take $f\in I$ and consider $\hom_w(f)\in \hom_w(I)$. By definition $f=\pi(F)$ for some $F\in J$. Since $J$ is homogeneous, $F=\sum_iF_i$ where $F_i\in J$ is homogeneous of degree $i$. If $d=\max\{i:F_i\neq 0\}$, we can replace $F$ with $F'=\sum_iz^{d-i}F_i$, which is a homogeneous element of $J$ such that $\pi(F')=f$. So we can assume at once that $F$ is homogeneous.
As before, we can write $F=z^{r}G$ where $G$ is a homogeneous polynomial of $S[z]$ not divided by $z$. Since $R$ is flat over $K[t]$, $t$ is a nonzero-divisor on $R$, so that $G$ belongs to $J$. So $\hom_w(f)=G$ belongs to $J$. Since $\hom_w(I)$ is generated by elements of the form $\hom_w(f)$ with $f\in I$, we conclude that $\hom_w(I)\subset J$. 
\end{proof}

In view of the previous lemma, we will refer to a $\NN$-graded ring $R$ which is a $K[t]$-module without nontrivial torsion, such that $t\in R$ is homogeneous of degree 1 and such that $R_0=K$, as a Gr\"obner deformation. We introduce the following concept:

\begin{definition}
Let $\Pp$ be some property that can have a ring. We say that $\Pp$ is $G$-deforming if the following two conditions hold: 
\begin{enumerate}
\item If $R$ is a $\NN$-graded ring with $R_0=K$ and $x\in R$ is a nonzero-divisor on $R$ of degree 1 such that $R/xR$ has property $\Pp$, then $R_x$ has property $\Pp$ as well.
\item If $R$ is a (not necessarily graded) ring such that $R[X,X^{-1}]$ has property $\Pp$, then $R$ has property $\Pp$ as well.
\end{enumerate}
\end{definition} 

\begin{proposition}\label{p:G-def}
Let $R$ be a Gr\"obner deformation. If $R/tR$, has a G-deforming property $\Pp$, then $R/(t-\lambda)R$ has property $\Pp$ as well for each $\lambda\in K$.

In other words, if $I\subset S=K[X_1,\ldots ,X_n]$ is an ideal and $w\in\NN^n$ a weight vector such that $S/\In_w(I)$ has a $G$-deforming property $\Pp$, then $S/I$ has property $\Pp$ as well.
\end{proposition}
\begin{proof}
Since $\Pp$ is $G$-deforming, then:
\begin{itemize}
\item Because $R/tR$ has property $\Pp$, then $R_t$ has property $\Pp$ as well.
\item Notice that, if $A=R/(t-1)R$, $R_t\cong A[X,X^{-1}]$. Since $A[X,X^{-1}]$ has property $\Pp$, $A=R/(t-1)R$ has property $\Pp$ as well.
\end{itemize}
So $R/(t-1)R$ has property $\Pp$. Now simply notice that, since $R$ is a Gr\"obner deformation, $R/(t-\lambda)R$ is isomorphic to $R/(t-1)R$ for $\lambda\in K\setminus\{0\}$.
\end{proof}

\begin{example}
The conclusion of Proposition \ref{p:G-def} may fail for 1-parameter affine deformations over $K$ which are not Gr\"obner. For example, if $K$ is a field of characteristic 5, consider 
\[R=K[X,Y,Z,t]/(tX^3+tY^3+tZ^3+XYZ).\]
such an $R$ is a 1-parameter affine deformation over $K$ and $R/tR\cong K[X,Y,Z]/(XYZ)$ is strongly $F$-injective. As we will see, being strongly $F$-injective is a $G$-deforming property, however $R/(t-1)R\cong K[X,Y,Z]/(X^3+Y^3+Z^3+XYZ)$ is not even $F$-injective.
\end{example}

\section{$F$-singularities and Gr\"obner deformations}

Let $p$ be a prime number. Let $R$ be a ring of characteristic $p$, and consider the \emph{Frobenius map}:

\begin{eqnarray*}
F: & R\longrightarrow R \\
& r\mapsto r^p
\end{eqnarray*}

Note that $F$ is a ring homomorphism. We denote by $F_*R$ the $R$-module defined as follows:  
\begin{itemize}
\item $F_*R = R$ as additive group;
\item $r\cdot x = r^px$ for all $r\in R$ and $x\in F_*R$.
\end{itemize}
This way we can also think of $F$ as the following map of $R$-modules:
\begin{eqnarray*}
F: & R\longrightarrow F_*R \\
& r\mapsto r^p
\end{eqnarray*}

The ring $R$ is reduced if and only if $F$ is injective, so it is natural to introduce the following concept:

\begin{definition}
$R$ is \emph{$F$-split} if there exists a homomorphism $\theta:F_*R\rightarrow R$ of $R$-modules such that $\theta\circ F=1_R$. Such a $\theta$ is called an $F$-splitting of $R$. 
\end{definition}

If $I$ is an ideal of $R$, we have an induced map of $R$-modules $F:H^i_I(R)\to H^i_I(F_*R)$ for all $i\in\NN$. As Abelian groups, it is easy to check that $H^i_{F(I)R}(R)=H^i_I(F_*R)$, hence, since $F(I)R=(x^p:x\in I)$ and $I$ have the same radical, we have a map of Abelian groups:
\[F:H_I^i(R)\to H_I^i(R).\]
If $R$ is $F$-split, of course $F:H^i_I(R)\to H^i_I(F_*R)$ splits as a map of $R$-modules. In particular, $F:H^i_I(R)\to H^i_I(R)$ is injective for any ideal $I\subset R$ and $i\in\NN$. The latter fact turned out to be very powerful since the work of Hochster and Roberts \cite{Hochster-Roberts}, so it has been natural to introduce the following definition:

\begin{definition}
$R$ is \emph{$F$-injective} if the map $F:H^i_{\mm}(R)\to H^i_{\mm}(R)$ is injective for any maximal ideal $\mm\subset R$ and $i\in\NN$.  
\end{definition}

The ``$F$-split'' property does not deform, i.e. there are examples of local rings $R$ which are not $F$-split but such that $R/xR$ is $F$-split for some regular element $x\in R$ (see Example \ref{e2:anurag}). It is still an open problem whether the ``$F$-injective'' property deforms. For this reason we further need to introduce the following property:

\begin{definition}
$R$ is \emph{$F$-full} if the image of the map $F:H^i_{\mm}(R)\to H^i_{\mm}(R)$ generates $H^i_{\mm}(R)$ as $R$-module for any maximal ideal $\mm\subset R$ and $i\in\NN$.  
\end{definition}

It turns out that, if $R$ is $F$-split, then it is $F$-full (\cite[Theorem 3.7]{Ma} and \cite[Remark 2.4]{Ma-Qui}). Moreover the ``$F$-full property'' is known to deform (\cite[Theorem 4.2]{Ma-Qui}). Since there is no relationship between being $F$-full and being $F$-injective (any Cohen-Macaulay ring is $F$-full; on the other hand there exist $F$-injective rings that are not $F$-full, see \cite[Example 3.5]{MSS}), we introduce the last $F$-singularity of this paper:

\begin{definition}
$R$ is \emph{strongly $F$-injective} if it is $F$-injective and $F$-full.  
\end{definition}

By the previous discussion it follows that being strongly $F$-injective is a property in between the ```$F$-split'' and the ``$F$-injective'' properties. The important point for us is that the ``strongly $F$-injective'' property deforms by \cite[Corollary 5.16]{Ma-Qui}.

%
%

\subsection{$F$-splittings of the polynomial ring}\label{sec:splittingpoly}

In this subsection, we essentially combine parts of classical Fedder's paper \cite{Fedder} with parts of the more recent paper of Knutson \cite{Knutson}, in order to find interesting ideals having a squarefree Gr\"obner degeneration.

For this subsection, $K$ will be a perfect field of prime characteristic $p$ and $S=K[X_1,\ldots ,X_n]$ the polynomial ring in $n$ variables over $K$. It is easy to see that $F_*S$ is the free $S$-module generated by the monomials $X_1^{i_1}\cdots X_n^{i_n}$ with $i_j<p$ for all $j$. In particular, $S$ is $F$-split. We want to describe all the $F$-splittings $\theta:F_*S\rightarrow S$, and more generally the elements of $\Hom_S(F_*S,S)$. Of course the latter is a free $S$-module generated by the dual basis of $X_1^{i_1}\cdots X_n^{i_n}$ with $i_j<p$ for all $j$, say $\phi_{i_1,\ldots ,i_n}$. But our purpose is to understand the structure of $\Hom_S(F_*S,S)$ as an $F_*S$-module.

To this goal, let us introduce the fundamental element $\Tr:=\phi_{p-1,p-1,\ldots ,p-1}\in \Hom_S(F_*S,S)$. We claim that $\Hom_S(F_*S,S)$, as an $F_*S$-module, is generated by $\Tr$. More precisely, the following is an isomorphism of $F_*S$-modules:
\begin{eqnarray*}
\Phi:& F_*S\rightarrow & \Hom_S(F_*S,S)\\
& f \mapsto & f\star \Tr:g\mapsto \Tr(fg)
\end{eqnarray*}
The fact that $\Phi$ is an injective map of $F_*S$-modules is clear. For the surjectivity, just notice that, if $i_1,\ldots ,i_n$ are natural numbers such that $i_j<p$ for all $j$, we have $\phi_{i_1,\ldots ,i_n}=X_1^{p-i_1-1}\cdots X_n^{p-i_n-1}\star \Tr$. 

\begin{remark}
Notice that, given $f\in S$, $f\star \Tr$ is an $F$-splitting of $R$ if and only if the following two conditions hold:
\begin{enumerate} 
\item $X_1^{p-1}\cdots X_n^{p-1}\in \supp(f)$ and its coefficient in $f$ is 1.
\item If $X_1^{u_1}\cdots X_n^{u_n}\in \supp(f)$ and $u_1\equiv \ldots\equiv u_n \equiv -1$ (mod $p$), then $u_i=p-1$ $\forall \ i$.
\end{enumerate}
\end{remark}

\begin{definition}
If $\theta:F_*S\rightarrow S$ is an $F$-splitting, we say that an ideal $I\subset S$ is {\it compatibly split} with respect to $\theta$ if $\theta(I)\subset I$. 
\end{definition}

\begin{remark}
Of course, if an ideal $I\subset S$ is compatibly split with respect to an $F$-splitting $\theta$, then $\overline{\theta}:(F_*S)/I=F_*(S/I)\rightarrow S/I$ defines an $F$-splitting of $S/I$; in particular $S/I$ is $F$-split.
Furthermore, in this case, $\theta(I)=I$ (indeed the inclusion $I\subset \alpha(I)$ holds true for any $F$-splitting $\alpha\in\Hom_S(F_*S,S)$).
\end{remark}

\begin{proposition}\label{standardsplitting}
The map $\theta=X_1^{p-1}\cdots X_n^{p-1}\star \Tr\in \Hom_S(F_*S,S)$ is an $F$-splitting of $S$, and the compatibly split ideals with respect to $\theta$ are exactly the squarefree monomial ideals of $S$.
\end{proposition}

\begin{proof}
That $\theta$ is an $F$-splitting is clear, and it is easy to check that a squarefree monomial ideal is compatibly split with respect to $\theta$. 

Viceversa, let $g=\sum_{i=1}^sa_i\mu_i\in I$, where $\mu_i=X^{u_{i1}}_1\cdots X^{u_{in}}_n$ and $a_i\in K\setminus \{0\}$. Pick $i\in\{1,\ldots ,s\}$. Our purpose is to show that, if $I$ is compatibly split ideals with respect to $\theta$, then $\mu_i\in I$. Clearly, there exists $N\in\NN$ such that, for all $i\neq k\in \{1,\ldots ,s\}$, $u_{kj}\not\equiv u_{ij}$ (mod $p^N$) for some $j\in\{1,\ldots ,n\}$.
For each $j=1,\ldots ,n$, let $0\leq v_j<p^N$ such that $u_{ij}\equiv -v_j$ modulo $p^N$, and call $h=X^{v_1}_1\cdots X^{v_n}_ng\in I$. Since $I$ is a compatibly split ideal with respect to $\theta$, then $\theta^N(h)\in I$. Notice that the monomials in the support of $\theta^N(h)$ correspond to those $k\in\{1,\ldots ,s\}$ such that $u_{kj}\equiv u_{ij}$ modulo $p^N$ for all $j=1,\ldots ,n$. Hence $\theta^N(h)$ is a monomial, precisely 
\[\theta^N(h)=\sqrt[p^{\tiny N}]{a_iX_1^{u_{i1}+v_1}\cdots X_n^{u_{in}+v_n}}.\]
Since $\frac{u_{ij}+v_j}{p^N}\leq u_{ij}$ for any $j=1,\ldots ,n$, $\mu_i$ is a multiple of $\theta^N(h)\in I$, so that $\mu_i\in I$. This shows that $I$ is a monomial ideal. That $I$ is radical follows from the fact that $S/I$ is $F$-split.
\end{proof}

The following proposition has already been proved in \cite[Lemma 2]{Knutson}. We provide a proof here for the convenience of the reader.

\begin{proposition}\label{p:splitin}
Let $w=(w_1,\ldots ,w_n)\in (\NN_{>0})^n$ be a weight vector. Then, for any $g\in S$, either $\Tr(\init_w(g))=0$ or $\Tr(\init_w(g))=\init_w(\Tr(g))$. 
\end{proposition}

\begin{proof}
Given two vectors $(u_1,\ldots ,u_n), (v_1,\ldots ,v_n)\in \RR^n$ clearly we have:
\begin{align}\label{eq:splitin}
\sum_{i=1}^nu_iw_i\geq \sum_{i=1}^nv_iw_i \ & \Longleftrightarrow \ \sum_{i=1}^n\left(\frac{u_i+1}{p}-1\right)w_i\geq \sum_{i=1}^n\left(\frac{v_i+1}{p}-1\right)w_i. 
\end{align}
Recall that, if $\mu=X_1^{u_1}\cdots X_n^{u_n}$ is a monomial of $S$, $w(\mu)=\sum_{j=1}^nw_ju_j$, and that $w(f)=\max\{w(\nu):\nu\in\supp(f)\}$ for any $f\in R$.

Let $g=\sum_{i=1}^sa_i\mu_i\in I$, where $\mu_i\in\Mon(S)$ and $a_i\in K\setminus \{0\}$. Call $\mu_i=X_1^{u_{i1}}\cdots X_n^{u_{in}}$. If $\Tr(\init_w(g))\neq 0$, then there exists $i\in\{1,\ldots ,s\}$ such that $w(\mu_i)=w(g)$ and $u_{ij}\equiv -1$ (mod $p$) for all $j=1,\ldots ,n$.

Then $\Tr(\init_w(g))=\sum_{k\in A}\sqrt[p]{a_k}\Tr(\mu_k)$ where $A=\{k\in\{1,\ldots ,s\}:w(\mu_k)=w(g) \ \mbox{ and } \ u_{kj}\equiv -1 \ \forall \ j=1,\ldots ,n\}$. By our assumption $A$ is nonempty, indeed $i\in A$. On the other hand, $\Tr(g)=\sum_{k\in B}\sqrt[p]{a_k}\Tr(\mu_k)$ where $B=\{k\in\{1,\ldots ,s\}: u_{kj}\equiv -1 \ \forall \ j=1,\ldots ,n\}$. Of course $A\subset B\subset \{1,\ldots ,s\}$. Furthermore, using \eqref{eq:splitin}, $\{k\in B:w(\Tr(\mu_k)) \ \mbox{ is maximal}\}=\{k\in B:w(\mu_k) \ \mbox{ is maximal}\}=A$, so $\init_w(\Tr(g))=\sum_{k\in A}\sqrt[p]{a_k}\Tr(\mu_k)=\Tr(\init_w(g))$.
\end{proof} 

\begin{cor}\label{c:splitsqf}
Let $f\in S$ be such that there is a monomial order $<$ with $\init_<(f)=X_1^{p-1}\cdots X_n^{p-1}$. Then $f\star \Tr$ is an $F$-splitting of $S$, and $\init_<(I)\subset S$ is a squarefree monomial ideal for any compatibly split ideal (with respect to $f\star\Tr$) $I\subset S$.
\end{cor}

\begin{proof}
By Proposition \ref{standardsplitting}, it is enough to show that $\init_<(I)$ is a compatibly split ideal with respect to $X_1^{p-1}\cdots X_n^{p-1}\star \Tr$. Notice that $\init_<(I)$, as an $S$-submodule of $F_*S$, is generated by finitely many monomials, say $\mu_1,\ldots ,\mu_k$; so to check that $\init_<(I)$ is compatibly split with respect to $X_1^{p-1}\cdots X_n^{p-1}\star \Tr$ it is enough to check that $\Tr(X_1^{p-1}\cdots X_n^{p-1}\mu_i)\in \init_<(I)$ for all $i=1,\ldots ,k$. By definition of initial ideal, for any $i=1,\ldots ,k$ there are $g_i\in I$ such that $\init_<(g_i)=\mu_i$. Pick a weight vector $w\in(\NN_{>0}^n)$ such that $\init_w(g_i)=\init_<(g_i)$ for any $i=1,\ldots,k$ and $\init_w(f)=\init_<(f)$ (so $\init_w(I)=\init_<(I)$). Then either
$\Tr(X_1^{p-1}\cdots X_n^{p-1}\mu_i)=0$ or, using Proposition \ref{p:splitin},
\begin{align*}
\Tr(X_1^{p-1}\cdots X_n^{p-1}\mu_i)= & \Tr(\init_w(f)\init_w(g_i)) \\ 
 = & \Tr(\init_w(fg_i)) \\
 = & \init_w(\Tr(fg_i)))\in \init_w(\Tr(fI))\subset \init_w(I)=\init_<(I).
\end{align*}
\end{proof}

We end this subsection recalling the following useful criterion (see \cite[Lemma 1.6]{Fedder}).

\begin{proposition}\label{p:compatiblycriterion}
For any $f\in S$ and any ideal $I\subset S$, we have: 
\[(f\star\Tr)(I)\subset I \iff f\in I^{[p]}:I.\]
\end{proposition}

\subsection{Conclusions}

In this subsection we gather the conclusions we can get from the previous subsection. We will not assume anymore that $K$ is a perfect field of positive characteristic. It is useful to recall that, if $\phi:A\to B$ is a flat homomorphism of Noetherian rings, then for any two ideals $I,J\subset A$ one has (cf. \cite[Theorem 7.4]{matsu}:
\[(I\cap J)B=IB\cap JB, \ \ \ \ \ (I:J)B=IB:JB.\]

\begin{theorem}\label{t:charp}
Let $S=K[X_1,\ldots ,X_n]$ be the polynomial ring in $n$ variables over a field $K$ of characteristic $p>0$. Let $I\subset S$ be an ideal, $<$ a monomial order of $S$. If $\init_<(I^{[p]}:I)$ contains $X_1^{p-1}\cdots X_n^{p-1}$, then $\init_<(I)$ is a squarefree monomial ideal.
\end{theorem}
\begin{proof}
Let $f\in I^{[p]}:I$ such that $\init_<(f)=X_1^{p-1}\cdots X_n^{p-1}$. Let $K'$ be the perfect closure of $K$, $S'=S\otimes_KK'$ and $I'=IS'$. Since the inclusion $S\subset S'$ is flat, then $(I^{[p]}:I)S'=I'^{[p]}:I'$, so $(f\star\Tr)(I')\subset I'$ by Proposition \ref{p:compatiblycriterion}. So, by Corollary \ref{c:splitsqf} $\init_<(I')$ is a squarefree monomial ideal. Since the Buchberger algorithm is not affected by field extensions, we conclude that $\init_<(I)$ is a squarefree monomial ideal. 
\end{proof}

The proof of the next result is inspired by the results in \cite{seccia2020}, we recall that the $m$th symbolic power of an ideal $I$ of a Noetherian ring $R$ is the ideal $I^{(m)}:=I^m(T^{-1}R)\cap R$ where $T$ is the complement in $R$ of the union of the minimal prime ideals of $I$. In other words, $r\in I^{(m)}$ if and only if there exists $x\in R$ avoiding all the minimal prime ideals of $I$ such that $rx\in I^m$.

\begin{theorem}\label{t:symb}
Let $S=K[X_1,\ldots ,X_n]$ be the polynomial ring in $n$ variables over a field $K$ (not necessarily of positive characteristic). Let $I\subset S$ be an ideal, $<$ a monomial order of $S$, and call $h=\max\{\height(\pp):\pp\in\Mini(I)\}$. If $\init_<(I^{(h)})$ contains a squarefree monomial, then $\init_<(\sqrt{I})$ is a squarefree monomial ideal.
\end{theorem}
\begin{proof}
Of course we can assume that $I$ is radical, since $I^{(h)}\subset(\sqrt{I})^{(h)}$.

\bigskip

{\bf Let us first assume that $K$ has characteristic $p>0$}. Let $f\in I^{(h)}=\bigcap_{\pp\in\Mini(I)}\pp^{(h)}$ such that $\init(f)$ is a squarefree monomial. Of course we can assume $\init(f)=X_1\cdots X_n$, so that $\init(f^{p-1})=\init(f)^{p-1}=X_1^{p-1}\cdots X_n^{p-1}$, so if we show that $f^{p-1}\in I^{[p]}:I$ we are done by Theorem \ref{t:charp}. Pick $g\in I$. Then $g\in \pp$ for any minimal prime ideal $\pp$ of $I$. So let us see $g\in\pp S_{\pp}$. Since $S_{\pp}$ is a regular local ring of dimension $\height(\pp)\leq h$, $\pp S_{\pp}$ is generated by at most $h$ elements. Hence, since $f\in (\pp S_{\pp})^h$, by the pigeonhole principle, then 
\[f^{p-1}g\in (\pp S_{\pp})^{[p]}.\]
Hence there exists $a\in S\setminus \pp$ such that $af^{p-1}g\in\pp^{[p]}$. In particular $a^pf^{p-1}g\in\pp^{[p]}$, that is $f^{p-1}g\in\pp^{[p]}:a^p$. So, since $S$ is a regular ring, the Frobenius map $F:S\to S$ is flat by the theorem of Kunz (cf. \cite[Corollary 8.2.8]{BH93}) $f^{p-1}g\in(\pp:a)^{[p]}=\pp^{[p]}$, and
\[f^{p-1}g\in \bigcap_{\pp\in\Mini(I)}\pp^{[p]}=\left(\bigcap_{\pp\in\Mini(I)}\pp\right)^{[p]}=I^{[p]}.\]  
This concludes the proof if $K$ has positive characteristic.

\bigskip

{\bf If $K$ has characteristic 0}, let $\overline{K}$ denote the algebraic closure of $K$, $\overline{S}=\overline{K}[X_1,\ldots ,X_n]$ and $\overline{I}=I\overline{S}$. Since $K$, having characteristic 0, is perfect, $\overline{I}$ is a radical ideal. Moreover we have an equality of sets $\{\height(\pp):\pp\in\Mini(\overline{I})\}=\{\height(\pp):\pp\in\Mini(I)\}$: indeed, given a height $c$ prime ideal $\pp\subset S$, $\pp \overline{S}$ is a (perhaps not prime) ideal of $\overline{S}$ having all the minimal primes of height $c$, and the prime ideals of $\Mini(\overline{I})$ are minimal over some $\pp\overline{S}$ with $\pp\in\Mini(I)$. So $h=\max\{\height(\pp):\pp\in\Mini(\overline{I})\}$. Next, fix $f\in I^{(h)}$ such that $\init(f)$ is a squarefree monomial; so there exists $g\in S\setminus \left(\bigcup_{\pp\in\Mini(I)}\pp\right)$ such that $fg\in I^h$. Clearly, viewing $f$ and $g$ as polynomials of $\overline{S}$, we have $fg\in \overline{I}^h$. If $g$ were in some $\pp\in\Mini(\overline{I})$, then it would also belong to $\pp\cap S$, which is a minimal prime ideal of $I$, and we know this is not the case. So $fg\in \overline{I}^h$ and $g$ is not in $\bigcup_{\pp\in\Mini(\overline{I})}\pp$, so $f\in \overline{I}^{(h)}$. Therefore, {\bf we can assume that $K$ is algebraically closed}.

\vskip 2mm

Let $\{f_1,\ldots ,f_m\}$ be the reduced Gr\" obner basis of $I$ with respect to $<$. We need to show that $\init(f_i)$ is a squarefree monomial for each $i=1,\ldots ,m$. 

To this purpose, let us fix $f\in I^{(h)}$ such that $\init(f)$ is a squarefree monomial with coefficient 1 in $f$; so  there is $g\in S$ such that $fg\in I^h$ and $g$ does not belong to any of the minimal prime ideals of $I$. 
We can find a finitely generated $\ZZ$-algebra $Z\subset K$ such that the coefficients of $f,g$, those of all the $f_i$'s and those of the polynomials of the reduced Gr\"obner bases of the minimal prime ideals of $I$ are in $Z$. In particular, if $S_Z=Z[X_1,\ldots ,X_n]$ and $J_Z=J\cap S_Z$ for any ideal $J\subset S$, we have $I_ZS=I$ and $\pp_ZS=\pp$ for all $\pp\in\Mini(I)$. 

We also introduce the following notation: for all prime ideals $P\subset Z$, $Q(P)$ denotes the field of fractions of $Z/P$ (we write just $Q$ if $P$ is the zero ideal), $S_{Q(P)}=Q(P)[X_1,\ldots ,X_n]$ and $J_{Q(P)}$ stands for $J_ZS_{Q(P)}$ for any ideal $J\subset S$. Also, we will write $\overline{a}$ for the image in $S_{Q(P)}$ of an element $a\in S_Z$. Notice that a prime ideal  $P\subset Z$ contains at most one prime number $p\in\NN$; in this case, $Q(P)$ is a field of characteristic $p$.

\vskip 2mm

Notice that {\bf for any prime number $p>0$ and for all $P\in\Mini(pZ)$ we have that $\{\overline{f_1},\ldots ,\overline{f_m}\}$ is a (reduced) Gr\"obner basis of $I_{Q(P)}$}. This is simply because the coefficient of $\init(f_i)$ in $f_i\in S_Z$ is 1 for all $i=1,\ldots ,m$, so if the $S$-polynomials between the $f_i$'s reduce to zero modulo $\{f_1,\ldots ,f_m\}$ in $S$, they reduce to zero modulo $\{\overline{f_1},\ldots ,\overline{f_m}\}$ in $S_{Q(P)}$ as well.

\vskip 2mm

Similarly to above, notice that for any prime number $p>0$ and for all $P\in\Mini(pZ)$ we have that $\{\overline{g_1},\ldots ,\overline{g_k}\}$ is the reduced Gr\"obner basis of $\pp_{Q(P)}$ provided that $\{g_1,\ldots ,g_k\}$ is the reduced Gr\" obner basis of $\pp$, for any $\pp\in\Mini(I)$. In particular, {\bf for any prime number $p>0$ and for all $P\in\Mini(pZ)$ we have $\height(\pp)=\height(\pp_{Q(P)})$ and $\overline{g}\notin \pp_{Q(P)}$ for all $\pp\in\Mini(I)$}.

\vskip 2mm

We claim that {\bf there exists $N\in \NN$ such that, for all prime numbers $p>N$ and $P\in\Mini(pZ)$, we have that $I_{Q(P)}=\bigcap_{\pp\in\Mini(I)}\pp_{Q(P)}$}. The intersection of two polynomial ideals $A$ and $B$ of $S_{Q(P)}$ can be performed by computing a Gr\"obner basis of $At+B(1-t)\in S_{Q(P)}[t]$, so the claim follows by \cite[Lemma 2.3]{seccia2020} (see the proof of Proposition 2.2 in \cite{seccia2020} for the same application).

\vskip 2mm

We claim that, {\bf for all $\pp\in \Mini(I)$, there exists $N_{\pp}\in \NN$ such that, for all prime numbers $p>N_{\pp}$ and $P\in \Mini(pZ)$, $\pp_{Q(P)}$ is a prime ideal of $S_{Q(P)}$}.
To see this, consider the morphism of schemes 
\[\phi:X=\Spec(S_Z/\pp_Z)\to Y=\Spec(Z).\]
Notice that we have that $\phi$ is of finite type and $Y$ is irreducible. Since $\pp_QS=\pp_ZS$ is a prime ideal, then the special fibre $X_{\eta}$ ($\eta$ is the generic point of $Y$, namely the zero ideal of $Z$) is geometrically irreducible and geometrically reduced. Hence by Lemma 37.24.4 of \cite[\href{https://stacks.math.columbia.edu/tag/0574}{Tag 0574}]{stacks-project} and Lemma 37.25.5 of  \cite[\href{https://stacks.math.columbia.edu/tag/0553}{Tag 0553}]{stacks-project} there exists a nonempty open subset $U\subset Y$ such that $X_y$ is geometrically reduced and geometrically irreducible for all $y\in U$. In other words, $\pp_{Q(P)}$ is a geometrically prime ideal of $S_{Q(P)}$ for all prime ideals $P\in U$. We have proved the claim since there exists a nonzero ideal $J\subset Z$ such that $U=\{y\in Y: y\not\supset J\}$, so all but finitely many prime ideals of height 1 in $Z$ belong to $U$.

\vskip 2mm

Gathering everything, if we pick a prime number $p>\max\{N,N_{\pp}:\pp\in\Mini(I)\}$, we proved that any $P\in\Mini(pZ)$ is a prime ideal of $Z$ such that:
\begin{itemize}
\item $Q(P)$ is a field of characteristic $p>0$.
\item $I_{Q(P)}$ is a radical ideal with $\Mini(I_{Q(P)})=\{\pp_{Q(P)}:\pp\in \Mini(I)\}$.
\item The maximum height of a minimal prime ideal  of $I_{Q(P)}$ is $h$.
\item $\init((I_{Q(P)})^{(h)})$ contains a squarefree monomial.
\end{itemize} 
The above facts, and what previously proved in characteristic $p>0$, tell us that $\init(I_{Q(P)})$ is a squarefree monomial ideal. That is, $\init(f_i)$ is a squarefree monomial for all $i=1,\ldots ,m$, i.e. $\init(I)$ is a squarefree monomial ideal.
\end{proof}

\section{Some $G$-deforming $F$-singularities}

In this section we will prove that being $F$-rational or strongly $F$-injective are $G$-deforming properties. These facts depend on the fact that these properties deform in the local case by, respectively, \cite[Theorem 4.2(h)]{Hochster-Huneke} and \cite[Theorem 4.2(i)]{Ma-Qui}. We show that they also deform in the graded case accordingly with the nonlocal definitions, and to this purpose we proved Theorem \ref{F-rat'l-homogeneous} and Proposition \ref{p:strongly F-injective}, that are expected but we could not find in the literature. (We should point out that it would be possible to prove that $F$-rational or strongly $F$-injective are $G$-deforming properties in a more direct way, but we find Theorem \ref{F-rat'l-homogeneous} and Proposition \ref{p:strongly F-injective} interesting by themselves).

 A sequence of elements $x_1,\ldots, x_n$ in a ring $R$ are called \emph{parameters} if they can be extended to a  system of parameters in every 
 local ring $R_{\mathfrak{p}}$ of $R$ where $\mathfrak{p}$  is a prime ideal of $R$ that contains them. An ideal of $R$ is said to be a \emph{parameter ideal} if it can be generated by parameters.
 
If $R$ has prime characteristic $p$, the tight closure of an ideal $I\subset R$ is the ideal $I^*$ formed by the elements $r\in R$ such that there exists $c\in R\setminus\bigcup_{\mathfrak{p}\in \mathrm{Min}(R)}\mathfrak{p}$ such that $cr^{p^e}\in I^{[p^e]}=(x^{p^e}:x\in I)$ for any positive integer $e\gg 0$. We say that $I$ is tightly closed if $I=I^*$.

\begin{definition}
 A ring of prime characteristic is \emph{$F$-rational} if every parameter ideal is tightly closed.
\end{definition}
The following Lemma is well-known. For the convenience of the reader we include a proof.
\begin{lemma}
\label{F-rational-ff-map}
 Let $f:R \to S$ be a faithfully flat map. If $S$ is $F$-rational so is $R$.\\
\end{lemma}
\begin{proof}
 Since $R\to S$ is faithfully flat, then parameters of $R$ go to parameters of $S$. Let $I\subset R$ be a parameter ideal.
 Then $(IS)^*=IS$, as $S$ is $F$-rational. Now $I^*S\subseteq (IS)^*=IS$, hence $I^*=I^*S\cap R\subseteq IS\cap R=I$ (the equalities follow because $R\to S$ is faithfully flat). So $I^*=I$.
\end{proof}

\begin{theorem}
\label{F-rat'l-homogeneous}
If $R=\oplus_{i\in\ZZ}R_i$ is a $\ZZ$-graded ring having a unique maximal homogeneous ideal $\mm\subset R$, and $(R_0,\mm_0)$ is a complete local ring, then $R$ is $F$-rational if and only if $R_{\mm}$ is $F$-rational.
\end{theorem}
\begin{proof}

Notice that under the assumptions $R$ is a homomorphic image of a Cohen-Macaulay ring, so the ``only if'' direction follows from \cite[Theorem 4.2(f)]{Hochster-Huneke}. For the other direction we start by noting that $R$ is a domain since $R_{\mm}$ is a domain. 

First let us assume that $R_0$ is infinite. Then also the multiplicative group $R_0\setminus \mm_0$ is infinite. Since $R$ is reduced and finitely generated over the excellent local ring $R_0$, \cite[Theorem 3.5]{Velez} says that
\[
U=\{\pp\in \Spec R: R_{\pp}\textrm{ is $F$-rational}\}
\] 
is an open subset of $\Spec R$. Let $I\subset R$ be the radical ideal such that $V(I)=\{\pp\in \Spec R:\pp\supset I\}$ is the complement $\Spec R\setminus U$. Since $\mm\in U$, we are done once we show that $I$ is homogeneous. Consider the action $(R_0\setminus \mm_0)\times R\to R$ defined by $\lambda\cdot f=\lambda^df$ whenever $f\in R_d$, extended by additivity. Because $R_0\setminus \mm_0$ is infinite, it can be easily checked that $I$ is homogeneous if and only if it is stable under this action. Let us see that this is indeed true: of course $x\in I$ if and only if $R_x$ is $F$-rational. Because $\phi_{\lambda}:R\to R$ sending $f$ to $\lambda\cdot f$ is an automorphism of $R$ for all $\lambda\in R_0\setminus \mm_0$, $R_x$ is $F$-rational if and only if $R_{\lambda\cdot x}$ is $F$-rational for all $\lambda\in R_0\setminus \mm_0$. Hence, if $x\in I$, then $\lambda\cdot x\in I$ for any $\lambda\in R_0\setminus \mm_0$, and this concludes the proof in the case in which $R_0\setminus \mm_0$ is infinite. 

If $R_0$ is finite, then being a domain, it must be a perfect field. Hence $R_0\hookrightarrow L$ is a separable extension where $L$ is an algebraic closure of $R_0$. Consider $R'=R\otimes_{R_0} L$. Then $R_{\mm}\to R'_{\mathfrak{n}}$, where $\mathfrak{n}$ is the only maximal homogeneous ideal of $R'$, is a faithfully flat
smooth extension, so by \cite[Theorem 3.1]{Velez}, $R'_{\mathfrak{n}}$ is $F$-rational. Now, since $L=R'_0$ is infinite, by what has been previously said, $R'$ is $F$-rational, and therefore $R$ is $F$-rational by Lemma~\ref{F-rational-ff-map}.

\end{proof}

\begin{proposition}\label{p:F-rat'l}
Being $F$-rational is a $G$-deforming property.
\end{proposition}
\begin{proof}
Let $R$ be a $\NN$-graded ring with $R_0=K$, and suppose that $x\in R$ is a nonzero-divisor on $R$ of degree 1 such that $R/xR$ is $F$-rational. If $\mathfrak{m}=\bigoplus_{i\geq 1}R_i$, then $R_{\mathfrak{m}}/xR_{\mathfrak{m}}$ is $F$-rational. So $R_{\mathfrak{m}}$ is $F$-rational by \cite[Theorem 4.2(h)]{Hochster-Huneke} and hence $R$ is $F$-rational by Theorem~\ref{F-rat'l-homogeneous}. By \cite[Proposition 10.3.10]{BH93}, $R_x$ is $F$-rational. This proves condition (1) of the $G$-deforming definition. Condition (2) of the $G$-deforming definition follows by Lemma \ref{F-rational-ff-map}. 
\end{proof}

\begin{cor}
\label{F-rat'l}
Let $S=K[X_1,\ldots ,X_n]$ be a polynomial ring over a field $K$ of positive characteristic and $w\in(\NN_{>0})^n$. 
If $I\subset S$ is an ideal such that $S/\In_w(I)$ is $F$-rational, then $S/I$ is $F$-rational.
\end{cor}
\begin{proof}
This follows from Propositions \ref{p:G-def} and \ref{p:F-rat'l}
\end{proof}

If in the above corollary we replace the word ``$F$-rational'' with ``$F$-regular'' (that is, in all the localizations of $R$, every ideal is tightly closed), the statement is false.

\begin{example}\label{ex:anurag}
Let $S=K[X_1,\ldots, X_5]$ where $K$ has characteristic $p>2$, and $I$ the ideal generated by the $2$-minors of the matrix:
\[\begin{pmatrix}
X_4^2+X_5^3 & X_3 & X_2 \\
X_1 & X_4^2 & X_3^4-X_2
\end{pmatrix}.\]
Note that, if $\deg(X_4)=3$, $\deg(X_1)=\deg(X_3)=6$, $\deg(X_2)=24$ and $\deg(X_5)=2$, the ideal $I$ is homogeneous. By \cite[Proposition 4.5]{Singh-F-regularity} $S/I$ is not $F$-regular. However, considering the weight vector $w=(6,24,6,3,1)$ of $(X_1,X_2,X_3,X_4,X_5)$, one has that $\init_w(I)$ is the ideal of 2-minors of the matrix:
\[\begin{pmatrix}
X_4^2 & X_3 & X_2 \\
X_1 & X_4^2 & X_3^4-X_2
\end{pmatrix}.\]
By \cite[Proposition 4.3]{Singh-F-regularity} $S/\init_w(I)$ is $F$-regular, so ``$F$-regularity'' is not a $G$-deforming property.
\end{example}

Next we want to prove that being ``$F$-full'' or ``strongly $F$-injective'' are $G$-deforming properties.

\begin{lemma}
\label{ff-F-full}
Let $f: R \to S$ be a faithfully flat map between homomorphic images of regular rings of prime characteristic. If $S$ is $F$-full, so is $R$.
\end{lemma}
\begin{proof}
 First note that since $f$ is faithfully flat the natural map $\Spec S\to \Spec R$ induced by $f$ is surjective.  We will show $R_{\mathfrak{m}}$ is $F$-full for every maximal ideal $\mathfrak{m}$ of $R$, that is equivalent to say that $R$ is $F$-full since $H_{\mathfrak{m}R_{\mathfrak{m}}}^i(R_{\mathfrak{m}})\cong H_{\mathfrak{m}}^i(R)$.
 
 Let $\mathfrak{m}$ be a maximal ideal of $R$. Let $\mathfrak{n}$ be a maximal ideal in $S$ containing $\mathfrak{m}S$.
 Then $R_{\mathfrak{m}} \to S_{\mathfrak{n}}$ is a flat local map. By hypothesis $S_{\mathfrak{n}}$ is $F$-full.
By \cite[Proposition 3.9, Corollary 2.2]{DDM}, $R_{\mathfrak{m}}$ is $F$-full.
\end{proof}

\begin{proposition}
\label{p:F-full}
 Let $R=S/I$ with $S$ is an $n$-dimensional regular ring of prime characteristic. Then $R$ is $F$-full iff the natural map $\Ext^i_{S}(R,S)\to H^i_{I}(S)$ is  injective for every $ i=0,\ldots ,n$. 
 
 In particular, if $R$ is a homomorphic image of a regular ring:
 \begin{itemize}
 \item the $F$-full locus $\{\mathfrak{p}\in \Spec R: R_{\mathfrak{p}}\textrm{ is F-full}\}$ is a Zariski open subset of $\Spec R$.
 \item If $R=\oplus_{i\in\ZZ}R_i$ is a $\ZZ$-graded ring having a unique maximal homogeneous ideal $\mm\subset R$, then $R$ is $F$-full if and only if $R_{\mm}$ is $F$-full.
 \end{itemize}
\end{proposition}

\begin{proof}
By definition $R$ is $F$-full if and only if $R_{\mm}$ is $F$-full for all maximal ideals $\mm\subset R$, or equivalently, if and only if $R_{\mm}$ is $F$-full for all maximal ideals $M\subset S$ containing $I$ and $\mm=M/I$. On the other hand, by \cite[Proposition 2.1, Corollary 2.2]{DDM}, $R_{\mm}$ is $F$-full if and only if the natural map
\[\Ext^i_{S}(R,S)_M\cong \Ext^i_{S_M}(R_{\mm},S_M)\to H^i_{IS_M}(S_M)\cong H^i_{I}(S)_M\] 
is  injective for every $i=0,\ldots ,n$. Clearly, the above maps are injective for all maximal ideals $M\subset S$ containing $I$ and for all $i=0,\ldots ,n$ if and only if the maps $\Ext^i_{S}(R,S)\to H^i_{I}(S)$ are injective for all $i=0,\ldots ,n$.

For the last part, just call $N_i$ the kernel of $\Ext^i_{S}(R,S)\to H^i_{I}(S)$. Then, for a prime ideal $\pp\in\Spec R$, $R_{\pp}$ is $F$-full if and only if $(N_i)_{\pp}=0$ for all $i=0,\ldots ,n$. Therefore the $F$-full locus of $R$ is $\Spec R\setminus\cup_{i=0}^n\mathrm{Supp} N_i$, that is open. Finally, in the graded case, $N_i$ is a graded $R$-module, so $(N_i)_{\mm}=0$ implies $N_i=0$.
\end{proof}

\begin{proposition}\label{p:strongly F-injective}
Being $F$-full or strongly $F$-injective are $G$-deforming properties.
\end{proposition}

\begin{proof}
 Let $(R,\mathfrak{m})$ be an $\mathbb{N}$-graded ring with $R_0=K$. Suppose that $R/xR$ is $F$-full (strongly $F$-injective) for some homogeneous element $x$ of degree $1$. Then $R_{\mathfrak{m}}/xR_\mathfrak{m}$ is $F$-full (strongly $F$-injective) by definition. Hence $R_{\mathfrak{m}}$ is $F$-full (strongly $F$-injective) by \cite[Theorem 4.2, Corollary 5.16]{Ma-Qui}, thus $R$ is $F$-full (strongly $F$-injective) by 
Proposition~\ref{p:F-full} and \cite[Theorem 5.12]{DaMu}. By \cite[Lemma 3.4]{DDM} and \cite[Theorem 3.3]{DaMu}, $R_x$ is $F$-full (strongly $F$-injective), so condition $(1)$ of $G$-deforming property definition is satisfied. Now condition (2) follows from Lemma~\ref{ff-F-full} and \cite[Theorem 3.9]{DaMu}.
\end{proof}
Thus similar to Corollary~\ref{F-rat'l}, we have
\begin{cor}
\label{strongly F-injective}

Let $S=K[X_1,\ldots ,X_n]$ be a polynomial ring over a field $K$ of positive characteristic and $w\in(\NN_{>0})^n$. 
If $I\subset S$ is an ideal such that $S/\In_w(I)$ is strongly $F$-injective, then $S/I$ is strongly $F$-injective.
\end{cor}

\begin{cor}
\label{imp}
Let $K$ be a field of characteristic $p>0$ and $<$ a monomial order on $S=K[X_1,\ldots ,X_n]$, $I\subset S$ an ideal of $S$ and $A\subset S$ a $K$-subalgebra. 
\begin{enumerate}
\item If $\In_<(A)$ is Noetherian and normal, then $A$ is $F$-rational.
\item If $\In_<(I)$ is radical, then $S/I$ is strongly $F$-injective, and so $F$-injective.
\end{enumerate}
\end{cor}
\begin{proof}
(1). If $f_1,\ldots ,f_m\in A$ are a Sagbi basis of $A$, that is $\In_<(A)=K[\In_<(f_1),\ldots ,\In_<(f_m)]$, it is easy to see that $A=K[f_1,\ldots ,f_m]$. It turns out that, if $J\subset P=K[Y_1,\ldots ,Y_m]$ is the kernel of the $K$-algebra homomorphism sending $Y_i$ to $f_i$ (so that $P/J\cong A$), there exists $u\in(\NN_{>0})^m$ such that $\In_u(J)$ is the kernel of the $K$-algebra homomorphism sending $Y_i$ to $\In_<(f_i)$ (hence $\In_u(J)$ is a binomial ideal and $P/\In_u(J)\cong \In_<(A)$).
Since $P/\In_u(J)\cong \In_<(A)$ is a normal toric ring, it is $F$-regular, being a direct summand of a polynomial ring (\cite[Proposition 4.12]{Hochster-Huneke}). In particular, $P/\In_u(J)$ is $F$-rational. So,
 by Proposition~\ref{F-rat'l}, $A\cong P/J$ is $F$-rational.

(2). Since $\In_<(I)$ is radical, it is generated by square free monomials, hence $S/\In_<(I)$ is $F$-split, in particular strongly $F$-injective, hence by Corollary~\ref{strongly F-injective}, $S/I$ is strongly $F$-injective.
\end{proof}

\begin{remark}
The conclusion that $A$ is $F$-rational if $\init_<(A)$ is normal was already proved in \cite[Corollary 2.3]{CHV}.
In the proof is used that a $\NN$-graded ring $R$ (with $R_0$ a field of positive characteristic) is $F$-rational whenever $R/xR$ is so for some non-zero divisor $x\in R_1$. It is not clear to us how to show this fact (certainly known in the local case) without using Theorem \ref{F-rat'l-homogeneous}.
 
The conclusion that $S/I$ is $F$-injective provided $\init_<(I)$ is a squarefree monomial ideal has been proved independently in \cite[Theorem 4.3]{Gonz}, where this result has been crucial to prove that the only Gorenstein binomial edge ideals are complete intersections. 
\end{remark}

\begin{example}\label{e2:anurag}
Let $S=K[X_1,\ldots, X_5]$ where $K$ has characteristic $p>3$, and $I$ the ideal generated by the $2$-minors of the matrix of Example \ref{ex:anurag}, namely:
\[\begin{pmatrix}
X_4^2+X_5^3 & X_3 & X_2 \\
X_1 & X_4^2 & X_3^4-X_2
\end{pmatrix}.\]
If $<$ is the lexicographic monomial order with $X_1>X_2>X_3>X_4>X_5$, then $\init_<(I)=(X_1X_3,X_1X_2,X_2X_3)$, so $S/I$ is strongly $F$-injective by Corollary \ref{imp}. However, $S/I$ is not $F$-split by \cite[Proposition 4.5]{Singh-F-regularity}. 
\end{example}

The following corollary can help in recognising certain classes of projective varieties whose defining ideal, in any embedding, cannot admit a squarefree initial ideal.

\begin{cor}
\label{imp2}
 Let $X$ be a projective scheme over a field $K$ of characteristic $0$ such that, for some embedding of $X$ in $\mathbb{P}^n$ and monomial order $<$ on $K[x_0,\cdots,x_n]$, we have that $\In_<(I)$ is squarefree (where $I$ is the defining ideal of the embedding). Then the Frobenius action on $H^i(X_p,\mathcal{O}_{X_p})$ must be injective for all $i>0$ and prime number $p\gg0$ ($X_p$ denotes a reduction mod $p$ of $X$).
\end{cor}

\begin{proof}
Let $X\simeq \Proj S/I$, where $S=K[x_0,\cdots,x_n]$ with respect to  some embedding of $X$ in $\mathbb{P}^n$ with  defining ideal $I=(f_1,\cdots, f_t)$. We can, and will, choose $f_1,\ldots ,f_t$ forming a Gr\"obner basis.
  Choose a finitely generated $\mathbb{Z}$-algebra $A$ in such a way that, taking $S_A=A[x_0,\cdots,x_n]$ and defining $I_A=(f_1,\cdots,f_t)S_A$, we have that $S_A/{I_A}$ is free over $A$ and $S_A/{I_A}\otimes_A K=S/I$. Let $X_A=\Proj S_A/I_A$.
Then a reduction modulo a prime number $p$ of $X$ has the form 
$X_p=\Proj S_A/I_A \otimes L$ where $L=A/P$ for a maximal ideal $P\subset A$ containing $p$. In particular $S_A/I_A\otimes_A L=S_p/I_p$, where $S_p=L[x_1,\cdots,x_n]$ where $L$ is a field of characteristic $p>0$ and $I_p=(\bar{f_1},\cdots,\bar{f_{t}})$. Furthermore, if $p$ is big enough, we can assume $\{\bar{f_1},\cdots,\bar{f_t}\}$ remains a Gr\"{o}bner basis of $I_p$ and $\overline{\In(f_i)}=\In(\bar{f_i})$ for all $i$.
Hence $\In(I_p)$ is also square free. Thus by Corollary~\ref{imp}, $S_p/I_p$ is $F$-injective for all $p\gg 0$. Since for all $i>0$,
$H^i(X_p,\mathcal{O}_{X_p})=[H^{i+1}_{\mathfrak{m}_p}(S_p/I_p)]_0$, where $\mathfrak{m}_p$ denotes the homogeneous maximal ideal of $S_p/I_p$, the Frobenius action on $H^i(X_p,\mathcal{O}_{X_p})$ is injective.
\end{proof}

\section{Examples}

For the convenience of the reader we recall briefly the definitions of {\it Algebra with Straightening Law (ASL)} and {\it Cartwright-Sturmfels ideals}. For more details see, respectively, \cite{BV} and \cite{CDG}.

\bigskip

\noindent {\bf ASL}. Let $A=\oplus_{i\in \NN} A_i$ be a $\NN$-graded algebra and let $(H,\prec)$ be a finite poset. 
 Let $H\to \cup_{i>0}A_i$ be an  injective function. The elements of $H$ will be identified with their images. Given a chain $h_1\preceq h_2 \preceq \dots  \preceq h_s$ of elements of $H$ the corresponding product   $h_1\cdots h_s\in A$ is called  standard monomial. 
 One says that $A$ is an ASL on $H$ (with respect to the given embedding $H$ into $\cup_{i>0}A_i$) if three conditions are satisfied: 
 \begin{itemize}
\item The elements of $H$ generate  $A$ as a $A_0$-algebra. 
 \item The standard monomials are $A_0$-linearly independent. 
 \item For every pair $h_1, h_2$ of incomparable elements of $H$ there is a relation 
 $$h_1h_2=\sum_{j=1}^u  \lambda_j h_{j1}\cdots h_{jv_j}$$ 
 where $\lambda_j\in A_0\setminus \{0\}$,  the  $h_{j1}\cdots h_{jv_j}$  are distinct standard monomials and, assuming that $h_{j1}\preceq \dots \preceq  h_{jv_j}$,   one has $h_{j1}\prec h_1$ and $h_{j1}\prec h_2$ for all $j$. 
 \end{itemize}
 
 \bigskip

\noindent {\bf Cartwright-Sturmfels ideals}. Given positive integers  $d_1,\dots, d_m$ one considers the polynomial ring  $S=K[X_{ij} : 1\leq i\leq m \mbox{ and }1\leq j\leq d_i]$ with $\ZZ^m$-graded structure induced by assignment $\deg(x_{ij})=e_i\in \ZZ^m$. The group $G=\GL_{d_1}(K)\times \cdots \times \GL_{d_m}(K)$ acts on $S$  as the group of multigraded $K$-algebra automorphisms. The Borel subgroup  $B=B_{d_1}(K)\times \cdots  \times B_{d_m}(K)$ of the upper triangular invertible matrices acts on $S$ by restriction.  An ideal $J$ is  Borel-fixed if $g(J)=J$ for all $g\in B$.  
A multigraded ideal $I\subset S$ is Cartwright-Sturmfels if its multigraded Hilbert function coincides with that of a Borel-fixed radical ideal. 

\begin{cor}\label{c:ASL}
Let $K$ be a field of characteristic $p>0$. Then the following $K$-algebras are strongly $F$-injective, and so $F$-injective:
\begin{enumerate}
\item Algebras with straightening law.
\item Quotients of the form $S/I$ where $S$ is a polynomial ring over $K$ and $I$ is a Cartwright-Sturmfels ideal (e.g. $I$ is a binomial edge ideal).
\end{enumerate}
\end{cor}
\begin{proof}
(1) Writing $A=S/I$ where $S$ is a polynomial ring in variables indexed by the poset $H$ over $K$, and $I$ is the ideal generated by the straightening relations, choosing a degree (according to the grading of $A$) reverse lexicographic order $<$ extending the partial order on $H$, it easily follows from the definition that $\init_<(I)$ is a quadratic squarefree monomial ideal, hence $A$ is strongly $F$-injective by Corollary \ref{imp}. 

(2) In \cite[Proposition 1.6]{CDG} it has been shown that, in this case, $\init_<(I)$ is a squarefree monomial ideal for any monomial order, so the thesis follows once again by Corollary \ref{imp}.
\end{proof}

\begin{remark}
Let $S=K[X_1,X_2,X_3,X_4]$, where $K$ is algebraically closed field of  characteristic $p>0$, and $I$ the ideal generated by the $2$-minors of the matrix:
\[\begin{pmatrix}
X_4^4 & X_1 & X_3 \\
X_2 & X_4^4 & X_2-X_3
\end{pmatrix}.\]
One notes that $I=(X_1X_2-X_4^8, ~ X_2X_3-X_4^4(X_2-X_3),~  X_1X_3-X_4^8+X_4^4X_3)$.
It is easy to check that the ring $S/I$  is an ASL on the poset $H$ below:
\begin{equation*}
   \begin{tikzpicture}
    \node  at (0,1) {$X_4$};
    \node at(-1,2){$X_3$};
    \node at(0,2){$X_2$};
    \node at(1,2){$X_1$};
    \draw[thick](0,1.3)--(0,1.7);
    \draw[thick](-0.2,1.3)--(-0.7,1.7);
    \draw[thick](0.2,1.3)--(0.7,1.7);
\end{tikzpicture}
\end{equation*}
that is, in the poset $H$ we have $X_4 < X_3,X_2,X_1$ ($X_1,X_2$ and $X_3$ are incomparable). By \cite[Example 7.15]{HoHu} $S/I$ is $F$-rational but not $F$-split.
In particular, there are algebras with straightening law that are not $F$-split. Notice that the poset $H$ is ``wonderful'' in the terminology of \cite{Ei}, and for an ASL $A$ with $A_0$ a complete local ring being $F$-split and $F$-pure are equivalent conditions, so this is a counterexample to a conjecture stated at page 245 of \cite{Ei}.

Similarly, there are Cartwright-Sturmfels ideals which are not $F$-pure. For example consider the binomial edge ideal of a pentagon, namely
\[I=(X_iY_{i+1}-X_{i+1}Y_i,X_5Y_1-X_1Y_5 \ : \ i=1,2,3,4)\subset S=K[X_1,\ldots ,X_5,Y_1,\ldots ,Y_5].\]
We have that $I$ is a Cartwright-Sturmfels ideal by \cite[Theorem 2.1]{CDG2}, however, if $K$ has characteristic 2, $S/I$ is not $F$-split  by \cite[Example 2.7]{Mat}.
\end{remark}

\begin{cor}
Let $K$ be a field of characteristic $p>0$. Then the following $K$-algebras are $F$-split:
\begin{itemize}
\item Gorenstein ASL.
\item Gorenstein quotients of the form $S/I$ where $S$ is a polynomial ring over $K$ and $I$ is a Cartwright-Sturmfels ideal
\end{itemize}
\end{cor}
\begin{proof}
For a Gorensteun ring being in $F$-split is equivalent to being $F$-injective by \cite[Lemma 3.3]{Fedder}. So the result follows from Corollary \ref{c:ASL}.
\end{proof}

%

The following argument has been suggested by Winfried Bruns.

\begin{cor}\label{c:At}
Let $M_t(X)$ be the set of $t$-minors of a $m\times n$ generic matrix $X$, and $K$ a field of characteristic $p>\min\{t,m-t,n-t\}$. The algebra of minors $K[M_t(X)]$ is $F$-regular.
\end{cor}
\begin{proof}
First of all, by \cite[Theorem 3.11]{BC1} there exists a monomial order such that $\In(K[M_t(X)])$ is a normal semigroup ring, so $K[M_t(X)]$ is $F$-rational by Corollary \ref{imp}.

In order to see that $K[M_t(X)]$ is $F$-regular, we assume that $m\leq n$. So, let us add $n-m$ rows to $X$ in order to form the generic $n\times n$-matrix $X'$. By \cite[Proposition 1.4]{BCV} and \cite[Proposition 4.12]{Hochster-Huneke}, if $K[M_t(X')]$ is $F$-regular, then $K[M_t(X)]$ is $F$-regular as well. Furthermore, by \cite[Proposition 1.3]{BCV}, $K[M_t(X')]$ is $F$-regular if and only if $K[M_{n-t}(X')]$ is $F$-regular.

So we can assume that $X$ is a generic $n\times n$-matrix and $t\geq n/2$. Now we can add $2t-n$ rows and $2t-n$ columns to $X$ and get a generic $2t\times 2t$-matrix $X'$. Again using \cite[Proposition 1.4]{BCV} and \cite[Proposition 4.12]{Hochster-Huneke}, if $K[M_t(X')]$ is $F$-regular, then $K[M_t(X)]$ is $F$-regular as well.

So, we can eventually assume that $X$ is a generic $2t\times 2t$-matrix. In this case, $K[M_t(X)]$ is Gorenstein by \cite[Theorem 5.5]{BC2}. Since a Gorenstein ring is $F$-regular if and only if it is $F$-rational by 	\cite[Corollary 4.7(a)]{Hochster-Huneke}, we are done.
\end{proof}

\begin{remark}
When $t=\min\{m,n\}$, the $K$-algebra $K[M_t(X)]$ is the coordinate ring of a Grassmannian in its Pl\"{u}cker embedding, and in this case the F-regularity in positive characteristic had already been proved in \cite[Theorem 7.14]{HoHu}.

In general, that the $K$-algebra $K[M_t(X)]$ is $F$-rational whenever $K$ a field of characteristic $p>\min\{t,m-t,n-t\}$ was already known and proved in \cite{BC3}. Analogously, in \cite{BC3} it has been proved that also the Rees algebra of the ideal of the $t$-minors of $X$ is  $F$-rational whenever $K$ a field of characteristic $p>\min\{t,m-t,n-t\}$. The $F$-split and $F$-regularity properties for these and other blowup algebras of determinantal objects are studied in \cite{DMN}.
\end{remark}

We conclude with the following corollary, recently proved in \cite[Theorem 4.3]{CNV}.

\begin{cor}
 If $X$ is a smooth projective curve of genus 1 over the rationals, then $\In_<(I)$ is never squarefree, where $I$ is the homogeneous ideal defining $X\subset\mathbb{P}^n$ (independently on the embedding). 
\end{cor}

\begin{proof}
Since $X$ is a smooth curve of genus $1$, then $X$ is isomorphic to an elliptic curve. Then for infinitely many primes $p$, the reduction mod $p$ considered in \cite{Elkies}, $X_p$ of $X$ is supersingular \cite[Theorem 1]{Elkies}. That is the Frobenius morphism on $H^1(X_p,\mathcal{O}_{X_p})$ is zero for infinitely many primes $p$. Hence the corollary follows from Corollary~\ref{imp2}.
\end{proof}

\bibliographystyle{plain}
\bibliography{Binomial-F-inj.bib}

\begin{thebibliography}{10}

\bibitem{BC3}
Winfried Bruns and Aldo Conca.
\newblock {$F$}-rationality of determinantal rings and their {R}ees rings.
\newblock {\em Michigan Math. J.}, 45(2):291--299, 1998.

\bibitem{BC1}
Winfried Bruns and Aldo Conca.
\newblock K{RS} and powers of determinantal ideals.
\newblock {\em Compositio Math.}, 111(1):111--122, 1998.

\bibitem{BC2}
Winfried Bruns and Aldo Conca.
\newblock Algebras of minors.
\newblock {\em J. Algebra}, 246(1):311--330, 2001.

\bibitem{BCV}
Winfried Bruns, Aldo Conca, and Matteo Varbaro.
\newblock Relations between the minors of a generic matrix.
\newblock {\em Adv. Math.}, 244:171--206, 2013.

\bibitem{BH93}
Winfried Bruns and J\"{u}rgen Herzog.
\newblock {\em Cohen-{M}acaulay rings}, volume~39 of {\em Cambridge Studies in
  Advanced Mathematics}.
\newblock Cambridge University Press, Cambridge, 1993.

\bibitem{BV}
Winfried Bruns and Udo Vetter.
\newblock {\em Determinantal rings}, volume 1327 of {\em Lecture Notes in
  Mathematics}.
\newblock Springer-Verlag, Berlin, 1988.

\bibitem{CDG}
A.~Conca, E.~De~Negri, and E.~Gorla.
\newblock Universal {G}r\"{o}bner bases and {C}artwright-{S}turmfels ideals.
\newblock {\em Int. Math. Res. Not. IMRN}, (7):1979--1991, 2020.

\bibitem{CDG2}
Aldo Conca, Emanuela De~Negri, and Elisa Gorla.
\newblock Cartwright-{S}turmfels ideals associated to graphs and linear spaces.
\newblock {\em J. Comb. Algebra}, 2(3):231--257, 2018.

\bibitem{CHV}
Aldo Conca, J\"{u}rgen Herzog, and Giuseppe Valla.
\newblock Sagbi bases with applications to blow-up algebras.
\newblock {\em J. Reine Angew. Math.}, 474:113--138, 1996.

\bibitem{CV}
Aldo Conca and Matteo Varbaro.
\newblock Square-free {G}r\"{o}bner degenerations.
\newblock {\em Invent. Math.}, 221(3):713--730, 2020.

\bibitem{CNV}
Alexandru Constantinescu, De~Negri Emanuela, and Matteo Varbaro.
\newblock Singularities and radical initial ideals.
\newblock {\em Bull. Lond. Math. Soc}, 52(4):674--686, 2020.

\bibitem{DDM}
Hailong Dao, Alessandro De~Stefani, and Linquan Ma.
\newblock {Cohomologically Full Rings}.
\newblock {\em International Mathematics Research Notices}, 10 2019.
\newblock rnz203.

\bibitem{DaMu}
Rankeya Datta and Takumi Murayama.
\newblock Permanence properties of {$F$}-injectivity, 2020.

\bibitem{DMN}
Alessandro De~Stefani, Jonathan Monta\~no, and Luis Nu\~nez Betancourt.
\newblock Blowup algebras of determinantal ideals in positive characteristic.
\newblock {\em Preprint}, 2021.

\bibitem{Ei}
David Eisenbud.
\newblock Introduction to algebras with straightening laws.
\newblock In {\em Ring theory and algebra, {III} ({P}roc. {T}hird {C}onf.,
  {U}niv. {O}klahoma, {N}orman, {O}kla., 1979)}, volume~55 of {\em Lecture
  Notes in Pure and Appl. Math.}, pages 243--268. Dekker, New York, 1980.

\bibitem{Elkies}
Noam~D. Elkies.
\newblock The existence of infinitely many supersingular primes for every
  elliptic curve over {${\bf Q}$}.
\newblock {\em Invent. Math.}, 89(3):561--567, 1987.

\bibitem{Fedder}
Richard Fedder.
\newblock {$F$}-purity and rational singularity.
\newblock {\em Trans. Amer. Math. Soc.}, 278(2):461--480, 1983.

\bibitem{Gonz}
Ren\'e Gonz\'alez-Mart\'inez.
\newblock Gorenstein binomial edge ideals.
\newblock {\em Math. Nach.}, 2021.

\bibitem{Hochster-Huneke}
Melvin Hochster and Craig Huneke.
\newblock {$F$}-regularity, test elements, and smooth base change.
\newblock {\em Trans. Amer. Math. Soc.}, 346(1):1--62, 1994.

\bibitem{HoHu}
Melvin Hochster and Craig Huneke.
\newblock Tight closure of parameter ideals and splitting in module-finite
  extensions.
\newblock {\em J. Algebraic Geom.}, 3(4):599--670, 1994.

\bibitem{Hochster-Roberts}
Melvin Hochster and Joel~L. Roberts.
\newblock The purity of the {F}robenius and local cohomology.
\newblock {\em Advances in Math.}, 21(2):117--172, 1976.

\bibitem{Knutson}
Allen Knutson.
\newblock Frobenius splitting, point-counting, and degeneration, 2009.

\bibitem{Ma}
Linquan Ma.
\newblock Finiteness properties of local cohomology for {$F$}-pure local rings.
\newblock {\em Int. Math. Res. Not. IMRN}, (20):5489--5509, 2014.

\bibitem{Ma-Qui}
Linquan Ma and Pham~Hung Quy.
\newblock Frobenius actions on local cohomology modules and deformation.
\newblock {\em Nagoya Math. J.}, 232:55--75, 2018.

\bibitem{MSS}
Linquan Ma, Karl Schwede, and Kazuma Shimomoto.
\newblock Local cohomology of {D}u {B}ois singularities and applications to
  families.
\newblock {\em Compos. Math.}, 153(10):2147--2170, 2017.

\bibitem{Mat}
Kazunori Matsuda.
\newblock Weakly closed graphs and {$F$}-purity of binomial edge ideals.
\newblock {\em Algebra Colloq.}, 25(4):567--578, 2018.

\bibitem{matsu}
Hideyuki Matsumura.
\newblock {\em Commutative ring theory}, volume~8 of {\em Cambridge Studies in
  Advanced Mathematics}.
\newblock Cambridge University Press, Cambridge, 1986.
\newblock Translated from the Japanese by M. Reid.

\bibitem{seccia2020}
Lisa Seccia.
\newblock Knutson ideals and determinantal ideals of hankel matrices, 2020.

\bibitem{Singh-F-regularity}
Anurag~K. Singh.
\newblock {$F$}-regularity does not deform.
\newblock {\em Amer. J. Math.}, 121(4):919--929, 1999.

\bibitem{stacks-project}
The {Stacks Project Authors}.
\newblock \textit{Stacks Project}.
\newblock \url{https://stacks.math.columbia.edu}, 2018.

\bibitem{Velez}
Juan~D. V\'{e}lez.
\newblock Openness of the {F}-rational locus and smooth base change.
\newblock {\em J. Algebra}, 172(2):425--453, 1995.

\end{thebibliography}
\end{document}